\documentclass[11pt,a4paper]{article}
\usepackage{amsmath,amssymb,amsthm,amsfonts,latexsym,graphicx,subfigure}

\usepackage[affil-it]{authblk}

\usepackage[numbers,sort&compress]{natbib}

\usepackage[a4paper,margin=3cm]{geometry}

\newcommand{\bigzero}{\mbox{\normalfont\Large\bfseries 0}}
\newcommand{\bigone}{\mbox{\normalfont\Large\bfseries 1}}
\newcommand{\rvline}{\hspace*{-\arraycolsep}\vline\hspace*{-\arraycolsep}}

\usepackage{tikz}
\usetikzlibrary{topaths,calc}
\usetikzlibrary{positioning}
\tikzset{bluenode/.style={circle,fill=gray!50,minimum size=0.4cm,inner sep=0pt},}
\tikzset{rednode/.style={circle,fill=black!100,minimum size=0.4cm,inner sep=0pt},}

\DeclareMathOperator{\tr}{tr}

\DeclareMathOperator{\id}{Id}

\DeclareMathOperator{\diag}{diag}

\newcommand{\starabove}{\overset{*}}
\usepackage{hyperref}

\theoremstyle{plain}
\newtheorem{theorem}{Theorem}[section]
\newtheorem{lemma}[theorem]{Lemma}
\newtheorem{proposition}[theorem]{Proposition}
\newtheorem*{proposition*}{Proposition}
\newtheorem{corollary}[theorem]{Corollary}

\theoremstyle{definition}
\newtheorem{definition}[theorem]{Definition}

\theoremstyle{remark}
\newtheorem{remark}{Remark}[section]

\begin{document}
\bibliographystyle{plain} 
\title{Spectral classes of hypergraphs}
\author{Raffaella Mulas}
\affil{Max Planck Institute for Mathematics in the Sciences\\ D--04103 Leipzig, Germany}
\date{}
\maketitle
\allowdisplaybreaks[4]

\begin{abstract}The notions of spectral measures and spectral classes, which are well known for graphs, are generalized and investigated for oriented hypergraphs.
\vspace{0.2cm}

\end{abstract}
\section{Introduction}In \cite{JJspectralclasses,Gu}, Gu, Jost, Liu and Stadler introduced the notion of \emph{spectral measure} in order to visualize the entire spectrum of the normalized Laplacian of a graph independently of the graph size. This allows one to define \emph{spectral distances} between graphs, as further investigated in \cite{spectraldistances}. Gu et al.\ also introduced and investigated \emph{spectral classes} with the aim of studying the asymptotics of spectra of growing families of graphs. Moreover, in \cite{rgg}, Lerario together with the author of this paper extended this theory and established results on spectral classes that also involve other operators associated to a graph, such as the degree matrix, the adjacency matrix and the Kirchhoff Laplacian.\newline

Here we further extend the theory of spectral measures and spectral classes to the case of \emph{oriented hypergraphs}: a generalization of classical hypergraphs in which a plus or minus sign is assigned to each vertex--hyperedge incidence. Oriented hypergraphs were introduced by Shi in \cite{Shi92}, while their corresponding adjacency, incidence and Kirchhoff Laplacian matrices were introduced by Reff and Rusnak in \cite{ReffRusnak}, and their corresponding normalized Laplacians were introduced in \cite{Hypergraphs} by Jost together with the author of this paper. The setting in which the authors introduced the normalized Laplacians in \cite{Hypergraphs} is even more general, as it concerns a generalization of oriented hypergraphs for which one can also assign \emph{both} a plus and a minus sign to a vertex--hyperedge incidence. Such hypergraphs are called \emph{chemical hypergraphs}, while a vertex that has both a plus and a minus sign for a hyperedge is called a \emph{catalyst}, and the terminology is motivated by modeling chemical reaction networks. As shown in \cite{Sharp}, however, if in the setting of chemical hypergraphs one chooses to define the degree of a vertex $i$ as the number of hyperedges containing $i$ not as a catalyst, as done in \cite{Sharp,MulasZhang,AndreottiMulas,pLaplacian}, without loss of generality one's study can be restricted to oriented hypergraphs when investigating the spectrum of the normalized Laplacian, and it is easy to see that this is also true for the degree matrix, the adjacency matrix and the Kirchhoff Laplacian. \newline

Spectral theory of oriented (and chemical) hypergraphs is gaining a lot of attention and we refer the reader to \cite{ReffRusnak,orientedhyp2013,orientedhyp2014,orientedhyp2016,orientedhyp2017,orientedhyp2018,orientedhyp2019,orientedhyp2019-2,orientedhyp2019-3,Hypergraphs,Master-Stability,Sharp,MulasZhang,AndreottiMulas,pLaplacian,andreotti2020eigenvalues} for a vast\,---\,but not complete\,---\,literature on this topic. However, to the best of our knowledge, spectral measures and spectral classes have not been yet investigated in this setting.\newline

\textbf{Structure of the paper.} In Section \ref{section:hyp} we recall the basic definitions concerning oriented hypergraphs and their known associated operators: the degree, adjacency, incidence, normalized Laplacian, hyperedge normalized Laplacian and Kirchhoff Laplacian matrices. We also introduce the \emph{hyperedge Kirchhoff Laplacian}. In Section \ref{section:measures} we define hypergraph spectral measures and spectral classes, and in Section \ref{section:families} we establish them for some given families of hypergraphs. Finally, in Section \ref{section:connected} we generalize two results proved in \cite{rgg} on the spectral classes of families of graphs that only differ by a fixed number of edges.

\section{Oriented hypergraphs and their operators}\label{section:hyp}
\begin{definition}
	An \emph{oriented hypergraph} is a triple $\Gamma=(V,H,\psi_\Gamma)$ such that $V$ is a finite set of vertices, $H$ is a finite multiset of elements $h\subseteq V$, $h\neq \emptyset$ called \emph{hyperedges}, while $\psi_\Gamma:V\times H\rightarrow \{-1,0,+1\}$ is the \emph{incidence function} and it is such that 
	\begin{equation*}
	    \psi_\Gamma(i,h)\neq 0 \iff i\in h.
	\end{equation*}A vertex $i$ is an \emph{input} (resp. \emph{output}) for a hyperedge $h$ if $\psi_\Gamma(i,h)=1$ (resp. $\psi_\Gamma(i,h)=-1$); two vertices $i\neq j$ are \emph{co-oriented in $h$} if $\psi_\Gamma(i,h)=\psi_\Gamma(j,h)\neq 0$ and they are \emph{anti-oriented in $h$} if $\psi_\Gamma(i,h)=-\psi_\Gamma(j,h)\neq 0$.
\end{definition}

\begin{remark}\label{rmk:graphs}
Simple graphs are oriented hypergraphs such that $H$ is a set and, for each $h\in H$, there exists a unique $i\in V$ such that $\psi_\Gamma(i,h)=1$ and there exists a unique $j\in V$ such that $\psi_\Gamma(i,h)=-1$.
\end{remark}

\begin{definition}
	We say that a hypergraph $\Gamma$ is \emph{bipartite} if one can decompose the vertex set as a disjoint union $V=V_1\sqcup V_2$ such that, for every hyperedge $h$ of $\Gamma$, either $h$ has all its inputs in $V_1$ and all its outputs in $V_2$, or vice versa (Fig. \ref{fig:bipartiteh}).
	\end{definition}

					\begin{figure}[h]
					\begin{center}
\begin{tikzpicture}
\node (v3) at (1,0) {};
\node (v2) at (1,1) {};
\node (v1) at (1,2) {};
\node (v6) at (5,0) {};
\node (v5) at (5,1) {};
\node (v4) at (5,2) {};

\begin{scope}[fill opacity=0.5]
\filldraw[fill=red!70] ($(v1)+(0,0.5)$) 
to[out=180,in=180] ($(v2) + (0,-0.5)$) 
to[out=0,in=180] ($(v5) + (0,-0.5)$)
to[out=0,in=0] ($(v4) + (0,0.5)$)
to[out=180,in=0] ($(v1)+(0,0.5)$);
\filldraw[fill=blue!70] ($(v2)+(0,0.5)$) 
to[out=180,in=180] ($(v3) + (0,-0.5)$) 
to[out=0,in=180] ($(v6) + (0,-0.5)$)
to[out=0,in=0] ($(v5) + (0,0.5)$)
to[out=180,in=0] ($(v2)+(0,0.5)$);
\end{scope}

\fill (v1) circle (0.05) node [right] {$1$} node [above] {\color{red}$+$};
\fill (v2) circle (0.05) node [right] {$2$} node [above] {\color{red}$+$} node [below] {\color{blue}$+$};
\fill (v3) circle (0.05) node [right] {$3$} node [below] {\color{blue}$+$};
\fill (v4) circle (0.05) node [left] {$4$} node [above] {\color{red}$-$};
\fill (v5) circle (0.05) node [left] {$5$} node [above] {\color{red}$-$}node [below] {\color{blue}$-$};
\fill (v6) circle (0.05) node [left] {$6$} node [below] {\color{blue}$-$};

\node at (0,2) {\color{red}$h_1$};
\node at (0,0) {\color{blue}$h_2$};
\end{tikzpicture}
					\end{center}
					\caption{A bipartite hypergraph with $V_1=\{1,2,3\}$ and $V_2=\{4,5,6\}$.}\label{fig:bipartiteh}
				\end{figure}
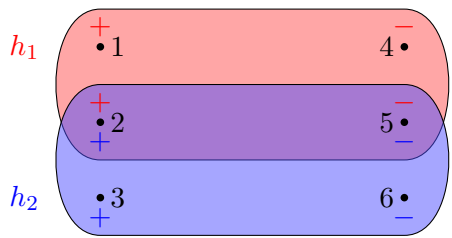

\begin{definition}
The \emph{degree} of a vertex $i$, denoted $\deg(i)$, is the number of hyperedges containing $i$. The \emph{cardinality} of a hyperedge $h$, denoted $\# h$, is the number of vertices that are contained in $h$.\newline

We say that a hypergraph is $p$--regular if $\deg(i)=p$ is constant for all $i\in V$. 
\end{definition}

From here on in the paper, we fix an oriented hypergraph $\Gamma=(V,H,\psi_\Gamma)$ on $n$ vertices $1,\ldots,n$ and $m$ hyperedges $h_1,\ldots, h_m$. For simplicity we assume that there are no vertices of degree zero.

\begin{definition}
The \emph{degree matrix} of $\Gamma$ is the $n\times n$ diagonal matrix $$D=D(\Gamma):=\diag(\deg(1),\ldots,\deg(n)).$$ The \emph{incidence matrix} of $\Gamma$ is the $n\times m$ matrix $$\mathcal{I}=\mathcal{I}(\Gamma):=(\psi_\Gamma(i,h))_{i\in V,h\in H}.$$
The \emph{adjacency matrix} of $\Gamma$ is the $n\times n$ matrix $A=A(\Gamma):=(A_{ij})_{ij},$ where $A_{ii}:=0$ for each $i\in V$ and, for $i\neq j$,
\begin{align*}
        A_{ij}:=& \# \{\text{hyperedges in which }i \text{ and }j\text{ are anti-oriented}\}\\
        &-\# \{\text{hyperedges in which }i \text{ and }j\text{ are co-oriented}\}.
\end{align*}The \emph{normalized Laplacian} of $\Gamma$ is the $n\times n$ matrix $$L=L(\Gamma):=\id-D^{-1/2}AD^{-1/2},$$ where $\id$ is the $n\times n$ identity matrix, while the \emph{hyperedge normalized Laplacian} of $\Gamma$ is the $m\times m$ matrix $$L^H=L^H(\Gamma):=\mathcal{I}^\top D^{-1}\mathcal{I}.$$ The \emph{Kirchhoff Laplacian} of $\Gamma$ is the $n\times n$ matrix $$K=K(\Gamma):=D-A.$$
\end{definition}

\begin{remark}
In \cite{Hypergraphs}, the normalized Laplacian is defined as $\hat{L}=\hat{L}(\Gamma):=\id-D^{-1}A$. Since $\hat{L}=D^{-1/2}LD^{1/2}$, the matrices $\hat{L}$ and $L$ are similar, implying that they have the same spectrum. Here we choose to work on $L$, a generalization of the Laplacian introduced by Chung in \cite{Chung} because, although $\hat{L}$ is a symmetric operator with respect to the usual scalar product (cf.\ Lemma~4.9 in \cite{Hypergraphs}), it is not necessarily a symmetric matrix. On the contrary, $L$ is a symmetric matrix for all $\Gamma$, and this allows us to apply the theory of symmetric matrices throughout the paper.
\end{remark}

As shown in \cite{Hypergraphs}, the normalized Laplacian has $n$ real, nonnegative eigenvalues, counted with multiplicity, while the hyperedge normalized Laplacian has $m$ eigenvalues counted with multiplicity, and the nonzero spectra of these two operators coincide. It is natural to ask how one could define, analogously, the \emph{hyperedge Kirchhoff Laplacian} as an $m\times m$ matrix that has the same nonzero spectrum as $K$. Motivated by \cite[Cor.~4.2]{orientedhyp2014}, we give the following definition.

\begin{definition}
The \emph{hyperedge Kirchhoff Laplacian} of $\Gamma$ is the $m\times m$ matrix $$K^H=K^H(\Gamma):=\mathcal{I}^\top \mathcal{I}.$$
\end{definition}
\begin{remark}
As pointed out in \cite{orientedhyp2014}, the Kirchhoff Laplacian can be rewritten as $K=\mathcal{I}\mathcal{I}^\top$, therefore it is immediate to see that $K$ and $K^H$ have the same nonzero spectra. Also, $K^H$ coincides with the Kirchhoff Laplacian of $\Gamma^*$, the \emph{dual hypergraph} of $\Gamma=(V,H,\psi_\Gamma)$, defined as $\Gamma^*:=(H,V,\psi_{\Gamma^*})$, where $\psi_{\Gamma^*}(h,i):=\psi_{\Gamma}(i,h)$.
\end{remark}
\begin{remark}\label{rmk:intervals}
We have mentioned already that the eigenvalues of $L$ and $L^H$ are real, as shown in \cite{Hypergraphs}. The same holds true also for $D$, $A$, $K$ and $K^H$, since these are all symmetric matrices. Furthermore, as shown in \cite{Hypergraphs} and \cite{Sharp}, the eigenvalues of $L$ and $L^H$ are contained in the interval $[0,n]$. Similarly, since the eigenvalues of $D$ are the vertex degrees, these are clearly contained in the interval $[1,m]$. As a consequence of \cite[Theorem~3.2]{orientedhyp2014}, the eigenvalues of $A$ are contained in the interval $[-mn,mn]$. Finally, as a consequence of Lemma~2.2 and Theorem~4.7 in \cite{orientedhyp2014}, the eigenvalues of $K$ (and therefore also the ones of $K^H$) are contained in the interval $[0,m(n+1)]$.
\end{remark}

\begin{remark}\label{rmk:multiplicity}Given a symmetric operator $Q$ and a real value $\lambda$, denote by $M_{\lambda}(Q)$ the multiplicity of $\lambda$ as eigenvalue of $Q$, with the convention that $M_{\lambda}(Q)=0$ provided $\lambda$ is not in the spectrum of $Q$. As observed in \cite{Hypergraphs}, the fact that $L$ and $L^H$ have the same nonzero spectra implies that
\begin{equation*}
    M_0(L)-M_0(L^H)=n-m.
\end{equation*}Moreover, by definition of $L$ and $K$, it is easy to see that $v$ is an eigenvector for $L$ with eigenvalue $0$ if and only if $v$ is an eigenvector for $K$ with eigenvalue $0$. Therefore, $M_0(L)=M_0(K)$. Since $K$ and $K^H$ have the same nonzero spectra, this also implies that $M_0(K^H)=M_0(L^H)$.
\end{remark}

\begin{remark}\label{rmk:regular}
It is easy to see that, if $\Gamma$ is $p$--regular, then
\begin{align*}
(\mathbf{v},\lambda) \text{ is an eigenpair for }L &\iff (\mathbf{v},p\lambda) \text{ is an eigenpair for }K\\
&\iff (\mathbf{v},p(1-\lambda)) \text{ is an eigenpair for }A.
\end{align*}
\end{remark}
\begin{remark}\label{rmk:trace}It is well known that the sum of the eigenvalues of a matrix equals its trace. Therefore, the eigenvalues of $L$ and $L^H$ sum to $n$, the eigenvalues of $A$ sum to $0$ and the eigenvalues of $D$, $K$ and $K^H$ sum to $\sum_{i\in V}\deg(i)$.
\end{remark}
\begin{definition}
Two hypergraphs $\Gamma$ and $\Gamma'$ on $n$ nodes are \emph{isospectral} with respect to an operator $Q$ if the matrices $Q(\Gamma)$ and $Q(\Gamma')$ are isospectral, i.e. they have the same eigenvalues, counted with multiplicity.
\end{definition}

\section{Spectral measures and spectral classes}\label{section:measures}
\begin{definition}
Given an $n\times n$ symmetric matrix $Q$ with eigenvalues
\begin{equation*}
		\lambda_1(Q)\leq\ldots\leq \lambda_n(Q),
\end{equation*}its \emph{spectral measure} is
\begin{equation*}
\mu(Q):=\frac{1}{n}\sum_{i=1}^n\delta_{\lambda_i(Q)},
\end{equation*}where $\delta$ denotes the Dirac measure.
\end{definition}

\begin{remark}
By Remark~\ref{rmk:intervals},
\begin{itemize}
    \item $\mu(L)$ is a probability measure on $[0,n]$;
    \item $\mu(D)$ is a measure on $[1,m]$;
    \item $\mu(A)$ is a measure on $[-mn,mn]$;
    \item $\mu(K)$ is a measure on $[0,m(n+1)]$.
\end{itemize}
\end{remark}

\begin{definition}Given a sequence $(Q_n)_{n}$ of $n\times n$ symmetric matrices and given a Radon measure $\rho$ on $\mathbb{R}$, $(Q_n)_{n}$ is said to belong to the \emph{spectral class} $\rho$ if
\begin{equation}\label{eq:spectralclass}
    \mu(Q_n)\rightharpoondown\rho\textrm{ as }n\rightarrow\infty,
\end{equation}where the weak convergence in \eqref{eq:spectralclass} means that, for every continuous function $f:\mathbb{R}\rightarrow\mathbb{R}$,
\begin{equation*}
    \mu(Q_n) (f)=\frac{1}{n}\sum_{i=1}^{n}f(\lambda_i(Q_n))\longrightarrow\rho(f), \text{ as }n\rightarrow\infty.
\end{equation*}
A sequence $(\Gamma_n)_{n}$ of hypergraphs on $n$ nodes is said to belong to the spectral class $\rho$ \emph{with respect to the operator $Q$} if $(Q(\Gamma_n))_{n}$ belongs to the spectral class $\rho$.
\end{definition}
From here on we fix the notations $D_n:=D(\Gamma_n)$, $A_n:=A(\Gamma_n)$, $L_n:=L(\Gamma_n)$ and $K_n:=K(\Gamma_n)$,  for a given hypergraph $\Gamma_n$.

\section{Spectral classes of given families of hypergraphs}\label{section:families}
In \cite{JJspectralclasses}, various examples of spectral classes for growing families of graphs with respect to the normalized Laplacian are computed. Here we investigate examples of general hypergraphs, with respect to various operators. We start with a simple example.
\begin{proposition}\label{prop:onehyperedge}
For each $n\in \mathbb{N}$, let $\Gamma_n$ be an oriented hypergraph on $n$ nodes with one single hyperedge of cardinality $n$. Then, $(\Gamma_n)_{n}$ belongs to the following spectral classes:
\begin{itemize}
 \item $\delta_1$, with respect to $D$ and $A$;
    \item $\delta_0$, with respect to $L$ and $K$.
\end{itemize}
\end{proposition}
\begin{proof}
Since all vertices have degree $1$ in $\Gamma_n$ for all $n$, it is clear that $D$ belongs to the spectral class $\delta_1$. Moreover, as shown in \cite{Hypergraphs}, $L_n$ has eigenvalues $0$ with multiplicity $n-1$ and $n$ with multiplicity $1$. Therefore, it is easy to see that $(\Gamma_n)_{n}$ belongs to the spectral class $\delta_0$ with respect to $L$. By Remark~\ref{rmk:regular}, since $\Gamma_n$ is $1$--regular for each $n$, also $K_n$ has eigenvalues $0$ with multiplicity $n-1$ and $n$ with multiplicity $1$, while $A_n$ has eigenvalues $1$ with multiplicity $n-1$ and $1-n$ with multiplicity $1$. Therefore, $K$ belongs to the spectral class $\delta_0$, while $A$ belongs to the class $\delta_1$. 
\end{proof}

\begin{remark}\label{rmk:graphs2}
For simplicity, from here on we consider examples of oriented hypergraphs $\Gamma$ in which $\psi_\Gamma$ has values in $\{0,+1\}$, that is, all vertices are inputs for all hyperedges in which they are contained. Note that such hypergraphs are not a generalization of graphs because, as we observed in Remark~\ref{rmk:graphs}, simple graphs are such that each edge has exactly one input and exactly one output. However, if $\Gamma=(V,H,\psi_\Gamma)$ is a graph and $\Gamma_+:=(V,H,\psi_{\Gamma_+})$ is a hypergraph with the same vertex set and the same hyperedge set as $\Gamma$, with the difference that $\psi_{\Gamma_+}:V\times H\rightarrow\{0,+1\}$, then $D(\Gamma)=D(\Gamma_+)$ and $A(\Gamma)=-A(\Gamma_+)$. Therefore,
\begin{equation*}
L(\Gamma_+)=\id-D(\Gamma_+)^{-1/2}A(\Gamma_+)D(\Gamma_+)^{-1/2}=\id+D(\Gamma)^{-1/2}A(\Gamma)D(\Gamma)^{-1/2}
\end{equation*}and
\begin{equation*}
    K(\Gamma_+)=D(\Gamma_+)-A(\Gamma_+)=D(\Gamma)+A(\Gamma),
\end{equation*}that is, $L(\Gamma_+)$ and $K(\Gamma_+)$ are the signless normalized Laplacian and the signless Kirchhoff Laplacian of $\Gamma$, respectively. In particular (cf.\ \cite[Remark~2.10]{AndreottiMulas}),
\begin{itemize}
    \item $\lambda$ is an eigenvalue of $A(\Gamma)$ if and only if $-\lambda$ is an eigenvalue of $A(\Gamma_+)$;
    \item $\nu$ is an eigenvalue of $L(\Gamma)$ if and only if $2-\nu$ is an eigenvalue of $L(\Gamma_+)$.
\end{itemize}Furthermore, if $\Gamma$ is $p$--regular, then $p$ is the only eigenvalue of $D(\Gamma)=D(\Gamma_+)$ and therefore $\nu$ is an eigenvalue for $K(\Gamma)$ if and only if $2p-\nu$ is an eigenvalue for $K(\Gamma_+)$.
\end{remark}
\begin{remark}
As shown in \cite[Prop.~4.4]{AndreottiMulas}, a hypergraph $\Gamma=(V,H,\psi_\Gamma)$ that is bipartite is isospectral (with respect to $D$, $A$, $L$, $L^H$, $K$, and $K^H$) to the hypergraph $\Gamma_+$ in the previous remark. Therefore, all the following examples still hold true if, instead of any of the following hypergraphs with only inputs, we consider bipartite hypergraphs with the same vertex set and the same hyperedge set.
\end{remark}

We now consider the example of complete hypergraphs, that were introduced in \cite{MulasZhang}.
\begin{definition}\label{def:complete}
We say that $\Gamma=(V,H,\psi_\Gamma)$ is the \emph{$r$--complete hypergraph}, for some $r\geq 2$, if $V$ has cardinality $n$, $H$ is given by all possible ${n \choose r}$ hyperedges of cardinality $r$, and $\psi_\Gamma:V\times H\rightarrow \{0,+1\}$, that is, all vertices are inputs for all hyperedges in which they are contained.
\end{definition}
\begin{lemma}\label{lemma:complete}
The $r$--complete hypergraph on $n$ nodes is such that:
\begin{itemize}
    \item The spectrum of $D$ is given by ${n-1 \choose r-1}$ with multiplicity $n$;
     \item The spectrum of $L$ is given by $\frac{n-r}{n-1}$ with multiplicity $n-1$ and $r$ with multiplicity $1$;
    \item The spectrum of $A$ is given by ${n-1 \choose r-1}\cdot \bigl(1-\frac{n-r}{n-1}\bigr)$ with multiplicity $n-1$ and ${n-1 \choose r-1}\cdot (1-r)$ with multiplicity $1$;
    \item The spectrum of $K$ is given by ${n-1 \choose r-1}\cdot \frac{n-r}{n-1}$ with multiplicity $n-1$ and ${n-1 \choose r-1}\cdot r$ with multiplicity $1$.
\end{itemize}
\end{lemma}
\begin{proof}
The first claim is trivial. As shown in \cite[Prop.~8.2]{AndreottiMulas}, the spectrum of $L$ in this case is given by $\frac{n-r}{n-1}$ with multiplicity $n-1$ and $r$ with multiplicity $1$. Since $\Gamma$ is ${n-1 \choose r-1}$--regular, by Remark~\ref{rmk:regular} this implies that:
\begin{itemize}
    \item The spectrum of $A$ is given by ${n-1 \choose r-1}\cdot \bigl(1-\frac{n-r}{n-1}\bigr)$ with multiplicity $n-1$ and ${n-1 \choose r-1}\cdot (1-r)$, with multiplicity $1$;
    \item The spectrum of $K$ is given by ${n-1 \choose r-1}\cdot \frac{n-r}{n-1}$ with multiplicity $n-1$ and ${n-1 \choose r-1}\cdot r$, with multiplicity $1$.
\end{itemize}

\end{proof}
\begin{corollary}\label{cor:complete}
Let $r\in \mathbb{N}$. For each $n\geq r$, let $\Gamma_n=(V_n,H_n,\psi_{\Gamma_n})$ be the $r$--complete hypergraph on $n$ nodes. Then, the sequence $(\Gamma_n)_{n\geq r}$ belongs to the spectral class $\delta_1$ with respect to $L$, while its spectral classes with respect to $D$, $A$ and $K$ do not exist.
\end{corollary}
\begin{proof}
By Lemma~\ref{lemma:complete}, it is clear that the weak limit of $\mu(Q_n)$ does not exist if $Q_n\in\{D_n,A_n,K_n\}$. Moreover, again by Lemma~\ref{lemma:complete}, for every continuous function $f:\mathbb{R}\rightarrow\mathbb{R}$
\begin{equation*}
    \mu(L_n) (f)=\frac{1}{n}\Biggl((n-1)\cdot f\left(\frac{n-r}{n-1}\right)+f(r) \Biggr)\xrightarrow{n\rightarrow \infty} f(1)=\delta_1(f),
    \end{equation*}that is, $(\Gamma_n)_{n\geq r}$ belongs to the spectral class $\delta_1$ with respect to $L$.
\end{proof}
\begin{remark}
From Corollary~\ref{cor:complete} together with Remark~\ref{rmk:graphs2}, we can re-deduce the fact that the sequence of complete graphs belongs to the spectral class $\delta_1$ with respect to $L$, as shown in \cite[Prop.~2.2]{JJspectralclasses}. Furthermore, it is interesting to note that, in the setting of Corollary~\ref{cor:complete}, the reason why the spectral classes with respect to $D$, $A$ and $K$ do not exist is that the corresponding eigenvalues tend to infinity and this is due, on its turn, to the fact that the vertex degrees tend to infinity. Hence, in general, it may be more convenient to work on the normalized Laplacian when studying spectral classes, as done in \cite{JJspectralclasses}, rather than on the other operators.
\end{remark}In view of the last observation, it is natural to ask what happens for growing families of hypergraphs for which the vertex degrees don't grow with $n$. An example is given by the following proposition. 
\begin{proposition}\label{prop:regular}
Let $p\in \mathbb{N}$ and let $(\Gamma_n)_n$ be a growing family of hypergraphs such that, for each $n$, $\Gamma_n$ is a $p$--regular hypergraph on $n$ nodes. Then, $(\Gamma_n)_n$ belongs to the spectral class $\delta_p$ with respect to $D$. Furthermore,
\begin{align*}
    (\Gamma_n)_n \text{ has a spectral class with respect to }L &\iff  (\Gamma_n)_n \text{ has a spectral class with respect to }K
    \\ &\iff (\Gamma_n)_n \text{ has a spectral class with respect to }A.
\end{align*}
\end{proposition}
\begin{proof}Since $\Gamma_n$ is a $p$--regular hypergraph for each $n$, the only eigenvalue of $D_n$ is $p$ (with multiplicity $n$) and therefore, in particular, $(\Gamma_n)_{n}$ belongs to the spectral class $\delta_p$ with respect to $D$. \newline

Now, if $(\Gamma_n)_{n}$ belongs to a spectral class $\rho$ with respect to $L$, for each for every continuous function $f:\mathbb{R}\rightarrow\mathbb{R}$ we have that
\begin{equation*}
    \frac{1}{n}\sum_{i=1}^{n}f(\lambda_i(L_n))\xrightarrow{n\rightarrow \infty}\rho (f).
\end{equation*}For a given continuous function $g:\mathbb{R}\rightarrow\mathbb{R}$, let $f(x):=g(p\cdot x)$ for any $x\in\mathbb{R}$. Then, by Remark~\ref{rmk:regular},
\begin{align*}
    &\frac{1}{n}\sum_{i=1}^{n}g(\lambda_i(K_n))=\frac{1}{n}\sum_{i=1}^{n}g(p\cdot \lambda_i(L_n))=\frac{1}{n}\sum_{i=1}^{n}f(\lambda_i(L_n))\\
    &\xrightarrow{n\rightarrow \infty}\rho (f)=\int_{\mathbb{R}}f(x)\textrm{d}\rho(x)=\int_{\mathbb{R}}g(p\cdot x)\textrm{d}\rho(x)=\int_{\mathbb{R}}g(p\cdot x)\textrm{d}\nu(p\cdot x)=\nu(g),
\end{align*}where $\nu$ is a measure such that $\rho(x)=\nu(p\cdot x)$. By the arbitrariness of $g$, we have shown that the existence of a spectral class with respect to $L$ implies the existence of a spectral class with respect to $K$. With a similar argument, one can show all other implications.
\end{proof}
\begin{corollary}\label{cor:cycle-graphs}
The sequence $(\Gamma_n)_{n}$ of cycle graphs on $n$ nodes belongs to the spectral class $\delta_2$ with respect to $D$. Also, it belongs to spectral classes $\rho$ that have no atoms with respect to $A$, $L$, and $K$, i.e. $\rho(B)=0$ for each finite subset $B\subset \mathbb{R}$.
\end{corollary}
\begin{proof}
By Proposition~\ref{prop:regular}, since each cycle is a $2$--regular graph, $(\Gamma_n)_{n}$ belongs to the spectral class $\delta_2$ with respect to $D$. Furthermore, in \cite[Prop.~2.5]{JJspectralclasses} it is shown that $(\Gamma_n)_{n}$ belongs to a spectral class with no atoms with respect to $L$. Together with Proposition~\ref{prop:regular}, this proves the claim also for $A$ and $K$.
\end{proof}
As next examples, we consider growing families of \emph{hyperflowers}: hypergraphs that were introduce in \cite{AndreottiMulas} and that generalize star graphs, up to forgetting the input/output structure. 
\begin{definition}The \emph{$l$-hyperflower with $t$ twins} (Fig. \ref{fig:hyperflowertwins}) is the oriented hypergraph $\Gamma=(V,H,\psi_{\Gamma})$ on $n$ vertices such that:
\begin{itemize}
    \item The vertex set $V$ can be decomposed as $V=\mathcal{C}\sqcup W$, where $\mathcal{C}$ is the \emph{core} and $W$ is given by the $t\cdot l$ \emph{peripheral vertices} $v_{11},\ldots,v_{1l},\ldots,v_{t1},\ldots,v_{tl}$;
    \item The hyperedge set is
    \begin{equation*}
        H=\{h|h=\mathcal{C}\cup\bigcup_{i=1}^{t} v_{ij} \,\mbox{ for } j=1,\ldots,l\};
    \end{equation*}
    \item $\psi_\Gamma:V\times H\rightarrow \{0,+1\}$, i.e. all vertices are inputs for all hyperedges in which they are contained.
\end{itemize}
\end{definition}
	\begin{figure}[h]
			\begin{center}
        \includegraphics[width=6cm]{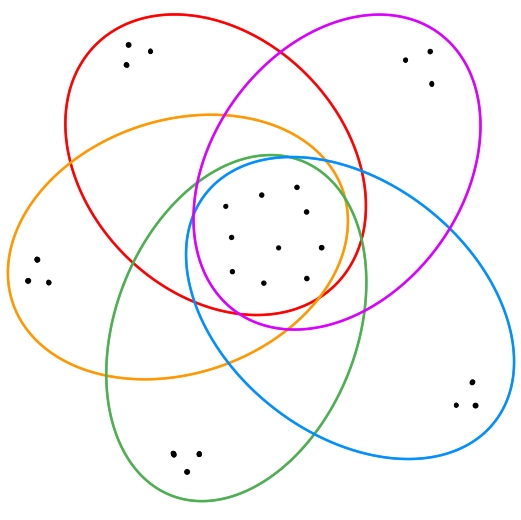}
        \caption{A $5$-hyperflower with $3$ twins.}        \label{fig:hyperflowertwins}
        \end{center}
    \end{figure}
In Proposition~\ref{prop:flowers1} below we consider a growing family of hyperflowers in which the core grows with $n$. In Proposition~\ref{prop:flowers2}, on the contrary, we consider a growing family of hyperflowers in which the peripheral vertices and the number of hyperedges grow. Before, we describe the spectral measures associated to hyperflowers.
\begin{theorem}\label{thm_spectrumflowers}
Let $\Gamma$ be the $l$--hyperflower with $t$ twins on $n$ vertices. Its associated spectral measures are:
\begin{itemize}
\item $\mu(D)=\frac{n-tl}{n}\cdot\delta_l+\frac{tl}{n}\cdot\delta_1$
\item $\mu(L)=\frac{n-l}{n}\cdot \delta_0+\frac{l-1}{n}\cdot \delta_t+\frac{1}{n}\cdot \delta_{n-tl+t}$
\item $\mu(K)=\frac{n-l}{n}\cdot \delta_0+\frac{l-1}{n}\cdot \delta_t+\frac{1}{n}\cdot \delta_{nl-tl^2+t}$
\item $\mu(A)=\frac{n-tl-1}{n}\cdot \delta_l+\frac{l(t-1)}{n}\cdot \delta_1+\frac{l-1}{n}\cdot\delta_{1-t}+\frac{1}{n}\cdot \delta_a+\frac{1}{n}\cdot \delta_b$,
\end{itemize}for some $a,b\in[-ln,ln]$.
\end{theorem}
\begin{proof}Since $\Gamma$ has $n-tl$ core vertices of degree $l$ and $tl$ peripheral vertices of degree $1$, the first claim is immediate. Moreover, as shown in \cite[Prop.~6.10]{AndreottiMulas}, the spectrum of $L$ in this case is given by:
\begin{itemize}
    \item $0$, with multiplicity $n-l$;
    \item $t$, with multiplicity $l-1$;
   \item $\lambda_n=n-tl+t$.\end{itemize}This proves the second claim.\newline
   
By Remark~\ref{rmk:multiplicity} and \cite[Prop.~6.10]{AndreottiMulas}, $K$ has eigenvalue $0$ with multiplicity $n-l$. Also, it is easy to see that, as in the case of $L$, the $l-1$ functions that are $1$ on the peripheral vertices of a fixed hyperedge, $-1$ on the peripheral vertices of another hyperedge and $0$ otherwise, are $l-1$ linearly independent eigenfunctions with eigenvalue $t$. Therefore, $t$ is an eigenvalue with multiplicity at least $l-1$. Since there is only one eigenvalue left, by Remark~\ref{rmk:trace} this is
\begin{equation*}
    \lambda=\sum_{i\in V}\deg(i)-t(l-1)=l(n-tl)+1(tl)-t(l-1)=nl-tl^2+t.
\end{equation*}
It is left to investigate the eigenvalues of $A$. Observe that, up to reordering the vertices,
\begin{equation*}
   A=-
\begin{pmatrix}
  \begin{matrix}
  0 & l & \ldots & l \\
  l & 0 &\ldots & l \\
  \vdots & \vdots & \ddots & \vdots\\
  l & l &\ldots & 0
  \end{matrix}
  & \rvline & \bigone_{(n-tl)\times t} & \rvline & \ldots & \rvline & \bigone_{(n-tl)\times t}\\
  \hline \bigone_{t\times (n-tl)} & \rvline & 
   \begin{matrix}
  0 & 1 & \ldots & 1 \\
  1 & 0 &\ldots & 1 \\
  \vdots & \vdots & \ddots & \vdots\\
  1 & 1 &\ldots & 0
  \end{matrix}
  & \rvline & \ldots & \rvline & \bigzero_{t\times t}\\
\hline \vdots & \rvline & \vdots & \rvline & \ddots & \rvline & \vdots \\
\hline \bigone_{t \times (n-tl)} & \rvline & \bigzero_{t\times t} & \rvline & \ldots  & \rvline & \begin{matrix}
  0 & 1 & \ldots & 1 \\
  1 & 0 &\ldots & 1 \\
  \vdots & \vdots & \ddots & \vdots\\
  1 & 1 &\ldots & 0
  \end{matrix}
\end{pmatrix}
\end{equation*}Therefore, the matrix $A-l\cdot \id $ is such that its first $n-tl$ rows (resp. columns) coincide, which implies that it has eigenvalue $0$ with multiplicity at least $n-tl-1$. Hence, $A$ has eigenvalue $l$ with multiplicity at least $n-tl-1$. Similarly, the fact that the matrix $A-\id$ has $l$ families of $t$ rows (resp. columns) that coincide, implies that $A-\id$ has eigenvalue $0$ with multiplicity at least $l(t-1)$, therefore $A$ has eigenvalue $1$ with multiplicity at least $l(t-1)$. Moreover, similarly to the cases of $L$ and $K$, it is easy to check that the $l-1$ functions that are $1$ on the peripheral vertices of a fixed hyperedge, $-1$ on the peripheral vertices of another hyperedge and $0$ otherwise, are $l-1$ linearly independent eigenfunctions with eigenvalue $1-t$. We have therefore listed $n-2$ eigenvalues of $A$ (with multiplicity). By Remark~\ref{rmk:intervals}, the two remaining eigenvalues must be in the interval $[-ln,ln]$. This proves the last claim.

\end{proof}
\begin{proposition}\label{prop:flowers1}
Fix $t,l\in\mathbb{N}$. For each $n \geq tl+1$, let $\Gamma_n$ be the $l$--hyperflower with $t$ twins on $n$ vertices. Then, $(\Gamma_n)_{n}$ belongs to the spectral class of $\delta_1$ with respect to $D$ and $A$, and it belongs to the spectral class of $\delta_0$ with respect to $L$ and $K$.
\end{proposition}
\begin{proof}
 It follows from Theorem~\ref{thm_spectrumflowers}.
\end{proof}
\begin{remark}
If we compare Proposition~\ref{prop:onehyperedge} and Proposition~\ref{prop:flowers1}, we can see that the spectral classes of these two families of hypergraphs are the same, with respect to all their corresponding operators. The intuition behind this is clear: a hyperflower with a growing core tends to look like a hypergraph that has only hyperedges of maximal cardinality.
\end{remark}
\begin{proposition}\label{prop:flowers2}
Fix $t,c\in \mathbb{N}$. For each $l\in \mathbb{N}$, let $\Gamma_{c+tl}$ be the $l$--hyperflower with $t$ twins on $c+tl$ vertices. Then, $(\Gamma_{c+tl})_{l\geq 1}$ belongs to the following spectral classes:
\begin{itemize}
    \item $\delta_1$, with respect to $D$;
    \item $\frac{t-1}{t}\cdot \delta_0+\frac{1}{t}\cdot \delta_t$, with respect to $L$ and $K$;
    \item $\frac{t-1}{t}\cdot \delta_1+\frac{1}{t}\cdot \delta_{1-t}$, with respect to $A$.
\end{itemize}
\end{proposition}
\begin{proof}
By Theorem~\ref{thm_spectrumflowers}, the spectral measures associated to $\Gamma_{c+tl}$ are:
\begin{itemize}
\item $\mu(D_{c+tl})=\frac{c}{c+tl}\cdot\delta_l+\frac{tl}{c+tl}\cdot\delta_1$
\item $\mu(L_{c+tl})=\frac{c+l(t-1)}{c+tl}\cdot \delta_0+\frac{l-1}{c+tl}\cdot \delta_t+\frac{1}{c+tl}\cdot \delta_{c+t}$
\item $\mu(K_{c+tl})=\frac{c+l(t-1)}{c+tl}\cdot \delta_0+\frac{l-1}{c+tl}\cdot \delta_t+\frac{1}{c+tl}\cdot \delta_{cl+t}$
\item $\mu(A_{c+tl})=\frac{c-1}{c+tl}\cdot \delta_l+\frac{l(t-1)}{c+tl}\cdot \delta_1+\frac{l-1}{c+tl}\cdot\delta_{1-t}+\frac{1}{c+tl}\cdot \delta_{a_l}+\frac{1}{c+tl}\cdot \delta_{b_l}$,
\end{itemize}for some $a_l,b_l\in[-l(c+tl),l(c+tl)]$. The claim follows by considering the limits.
\end{proof}
\begin{corollary}[Star graphs]
A growing family of star graphs belongs to the spectral class $\delta_1$ with respect to $D$, $L$ and $K$, and to the spectral class $\delta_0$ with respect to $A$.
\end{corollary}
\begin{proof}
It follows by letting $c=t=1$ in Proposition~\ref{prop:flowers2} and by Remark~\ref{rmk:graphs2}.
\end{proof}
\begin{remark}Instead of considering the spectral measures for the $n\times n$ matrices associated to hypergraphs, one could also look at the $m\times m$ operators $L^H$ and $K^H$. For instance, if $\Gamma$ is the $l$--hyperflower with $t$ twins on $n$ vertices, by Theorem~\ref{thm_spectrumflowers} and Remark~\ref{rmk:multiplicity} we have that
\begin{equation*}
    \mu(L^H)=\frac{l-1}{l}\cdot \delta_t+\frac{1}{l}\cdot \delta_{n-tl+t}
\end{equation*}and
\begin{equation*}
    \mu(K^H)=\frac{l-1}{l}\cdot \delta_t+\frac{1}{l}\cdot \delta_{n-tl^2+t}.
\end{equation*}Hence, in the setting of Proposition~\ref{prop:flowers1}, the growing family of hyperflowers (with growing core and fixed $l$) the spectral classes of $L^H$ and $K^H$ are not well defined. In the setting of Proposition~\ref{prop:flowers2}, the growing family of hyperflowers with growing number of hyperedges belongs to the spectral class of $\delta_t$ with respect to both $L^H$ and $K^H$.
\end{remark}

\section{Difference of spectral classes}\label{section:connected}
Theorem~2.8 in \cite{JJspectralclasses} states that, if two growing families of graphs $(\Gamma_{1,n})_{n}$ and $(\Gamma_{2,n})_{n}$ differ by at most a finite number $c$ of edges and their corresponding spectral measures with respect to $L$ have weak limits, then the two limits coincide. Similarly, Theorem~6.4 in \cite{rgg} states that the difference of the spectral measures of $(\Gamma_{1,n})_{n}$ and $(\Gamma_{2,n})_{n}$ goes to zero weakly, with respect to $A$, $D$, $K$ and $L$. In Section~\ref{section:weak} we generalize the latter result to the case of hypergraphs. Moreover, in Section~\ref{section:strong} we prove that a strong convergence with respect to the total variation distance holds in various cases.

\subsection{Weak convergence}\label{section:weak}
\begin{definition}
Two oriented hypergraphs $\Gamma_1$ and $\Gamma_2$ \emph{differ at most by $c$ hyperedges} if $\Gamma_1=(V,H_1,\psi_{H_1})$, $\Gamma_2=(V,H_2,\psi_{H_2})$ and there exist $\hat{H},\hat{H_1},\hat{H_2}\subseteq V$ such that:
\begin{itemize}
    \item $H_1=\hat{H}\sqcup \hat{H_1}$ and $H_2=\hat{H}\sqcup \hat{H_2}$
    \item $\psi_{H_1}\big|_{\hat{H}}=\psi_{H_2}\big|_{\hat{H}}$
    \item $\# (\hat{H_1}\cup\hat{H_2})\leq c$.
\end{itemize}
\end{definition}

We denote by $C_c(\mathbb{R},\mathbb{R})$ the space of continuous functions $f:\mathbb{R}\rightarrow \mathbb{R}$ with compact support.

	\begin{theorem}\label{thm:difference}
		Let $c_1,c_2\in \mathbb{N}$. Let $(\Gamma_{1,n})_n$ and $(\Gamma_{2,n})_n$ be two sequences such that, for every $n$, $\Gamma_{1,n}$ and $\Gamma_{2,n}$ are two hypergraphs on $n$ nodes that differ at most by $c_1$ hyperedges of cardinality at most $c_2$. Denote by $\mu_{1,n}$ and $\mu_{2,n}$ the spectral measures associated to one of the matrices $A$, $D$, $K$, $L$. Then
		\begin{equation}\label{eq:weakstar}
		\mu_{1,n}-\mu_{2,n}\starabove \rightharpoonup 0,
		\end{equation}where $\starabove \rightharpoonup$ denotes the weak star convergence, i.e. for each $f\in C_c(\mathbb{R},\mathbb{R})$
	\begin{equation*}
		\biggl|	\mu_{1,n}(f)-\mu_{2,n}(f)\biggr|\rightarrow 0.
			\end{equation*}
	\end{theorem}
	\begin{remark}We use the weak convergence in \eqref{eq:spectralclass} when considering the convergence of \emph{one sequence} of measures to a given measure, while we use the weak star convergence in \eqref{eq:weakstar} when considering the convergence of the difference between \emph{two sequences} of measures. Note that the condition for the weak convergence has to be satisfied by every continuous function $f:\mathbb{R}\rightarrow \mathbb{R}$, while the condition for the weak star convergence has to be satisfied by every such function with compact support. Therefore, the weak convergence is stronger than the weak star convergence.
	\end{remark}

	We need some preliminary definitions and results in order to prove Theorem~\ref{thm:difference}.

	\begin{definition}Given a real $n\times n$ symmetric matrix $Q$, its \emph{$1$-Schatten norm} is
		\begin{equation*}
		\|Q\|_{S^1}:=\sum_{i=1}^n|\lambda_i(Q)|
		\end{equation*}and its \emph{Frobenius norm} is
		\begin{equation*}
		\|Q\|_{F}:=\Biggl(\sum_{i,j=1}^n|Q_{ij}|^2\Biggr)^{1/2}=\sqrt{\tr (Q\cdot Q^\top)}.
		\end{equation*}

	\end{definition}
	The \emph{Weilandt-Hoffman inequality} \cite[Exercise 1.3.6]{Tao} holds:
		\begin{equation}\label{Weilandt-Hoffman inequality}
		\sum_{i=1}^n|\lambda_i(Q_1)-\lambda_i(Q_2)|\leq \|Q_1-Q_2\|_{S^1}.
		\end{equation}
	Moreover, since $Q$ is symmetric, we can write
		\begin{equation}\label{eq:QF}
		\|Q\|_{F}=\sqrt{\tr (Q^2)}=\Biggl(\sum_{i=1}^n\lambda_i(Q^2)\Biggr)^{1/2}.
		\end{equation}
	
Theorem~\ref{thmci} below generalizes Lemma~6.8 in \cite{rgg}.
	
	\begin{theorem}\label{thmci}Let $\Gamma_1=(V,H_1,\psi_{H_1})$ and $\Gamma_2=(V,H_2,\psi_{H_2})$ be two hypergraphs on $n$ nodes that differ by at most $c_1$ hyperedges of cardinality at most $c_2$. Let
	\begin{align*}
	    \Delta_1&:=A(\Gamma_1)-A(\Gamma_2)\\
	    \Delta_2&:=D(\Gamma_1)-D(\Gamma_2)\\
	    \Delta_3&:=K(\Gamma_1)-K(\Gamma_2)\\
	    \Delta_4&:=L(\Gamma_1)-L(\Gamma_2).
	\end{align*}Then,
	\begin{equation*}
	   \| \Delta_i\|_{S^1}\leq 3c_1^2\cdot c_2 \text{ for }i=1,2,3,
	\end{equation*}and
	\begin{equation*}
	     \| \Delta_4\|_{S^1}\leq 2\sqrt{2n}\cdot c_1\cdot c_2.
	\end{equation*}
		\end{theorem}
\begin{proof}The fact that $\Gamma_1$ and $\Gamma_2$ differ by at most $c_1$ hyperedges of cardinality at most $c_2$ implies that:
\begin{itemize}
    \item At most $c_1\cdot c_2$ vertices have different adjacencies in $\Gamma_1$ and $\Gamma_2$
    \item For each $i, j\in V$, $|(A_1)_{ij}-(A_2)_{ij}|\leq c_1$
    \item For each $i\in V$, $|\deg_{\Gamma_1}(i)-\deg_{\Gamma_2}(i)|\leq c_1$.
\end{itemize}Therefore,
\begin{itemize}
    \item $\Delta_1=A_1-A_2$ is a matrix with only zeros on the diagonal that has at most $2c_1\cdot c_2$ nonzero entries, whose absolute value is bounded by $c_1$;
        \item $\Delta_2=D_1-D_2$ is a diagonal matrix that has at most $c_1\cdot c_2$ nonzero entries, whose absolute value is bounded by $c_1$;
    \item $\Delta_3=K_1-K_2=(D_1-D_2)+(A_1-A_2)$ has at most $3c_1\cdot c_2$ nonzero entries, whose absolute value is bounded by $c_1$.
\end{itemize}Hence, for $j=1,2,3$, $\Delta_j$ has rank at most $c:=3c_1\cdot c_2$, and therefore at most $c$ nonzero eigenvalues. It follows that
\begin{align*}
			\|\Delta_j\|_{S^1}&=\sum_{i=n-c+1}^n|\lambda_i(\Delta_j)|\\
			&=\langle(1,\ldots,1),(|\lambda_{n-c+1}(\Delta_j)|,\ldots,|\lambda_n(\Delta_j)|) \rangle\\
		(\text{by Cauchy-Schwarz})\qquad	&\leq \sqrt{c}\cdot  \Biggl(\sum_{i=n-c+1}^n\lambda_i(\Delta_j)^2\Biggr)^{1/2} \\
		(\text{by } \eqref{eq:QF})\qquad	&=\sqrt{c}\cdot \|\Delta_j\|_F\\
			&=\sqrt{c}\cdot \biggl(\sum_{i,k}(\Delta_j)_{ik}^2\biggr)^{1/2}\\
			&\leq \sqrt{c}\cdot \bigl(c\cdot c_1^2\bigr)^{1/2}\\
			&=3c_1^2\cdot c_2.
			\end{align*}

Similarly,
\begin{equation*}
    \Delta_4:=L_1-L_2=D_1^{-1/2}A_1D_1^{-1/2}-D_2^{-1/2}A_2D_2^{-1/2}
\end{equation*}has entries with absolute value
\begin{align*}
    |(\Delta_4)_{ij}|&=\biggl|\frac{(A_1)_{ij}}{\sqrt{\deg_{\Gamma_1}(i)}\sqrt{\deg_{\Gamma_1}(j)}}-\frac{(A_2)_{ij}}{\sqrt{\deg_{\Gamma_2}(i)}\sqrt{\deg_{\Gamma_2}(j)}}\biggr|\\
    &\leq \biggl|\frac{(A_1)_{ij}}{\sqrt{\deg_{\Gamma_1}(i)}\sqrt{\deg_{\Gamma_1}(j)}}\biggr|+\biggl|\frac{(A_2)_{ij}}{\sqrt{\deg_{\Gamma_2}(i)}\sqrt{\deg_{\Gamma_2}(j)}}\biggr|\\
    &\leq 2.
\end{align*}Moreover, $\Delta_4$ has at most $c_1\cdot c_2$ nonzero rows (resp. columns), therefore it has rank at most $c_1\cdot c_2$. It follows that
\begin{align*}
			\|\Delta_4\|_{S^1}&=\sum_{i=n-c+1}^n|\lambda_i(\Delta_4)|\\
		&\leq \sqrt{c_1\cdot c_2}\cdot \Biggl(\sum_{i=n-c_1\cdot c_2+1}^n\lambda_i(\Delta_4)^2\Biggr)^{1/2} \\
	&=\sqrt{c_1\cdot c_2}\cdot \|\Delta_4\|_F\\
			&=\sqrt{c_1\cdot c_2}\cdot \biggl(\sum_{i,k}(\Delta_4)_{ik}^2\biggr)^{1/2}\\
			&\leq \sqrt{c_1\cdot c_2}\cdot \bigl(2n\cdot c_1\cdot c_2\cdot 4 \bigr)^{1/2}\\
			&=2\sqrt{2n}\cdot c_1\cdot c_2.
			\end{align*}

\end{proof}		
	We can now prove Theorem~\ref{thm:difference}.
\begin{proof}[Proof of Theorem~\ref{thm:difference}]
We first recall the statement of Proposition~6.7 in \cite{rgg}:
\begin{proposition*}[Prop.~6.7 in \cite{rgg}]
Let $Q_1,Q_2$ be two real $n\times n$ symmetric matrices such that 
			\begin{equation*}
			\|Q_1-Q_2\|_{S^1}\leq c.
			\end{equation*}Then, for each $f\in C_c(\mathbb{R},\mathbb{R})$ and for each $\varepsilon>0$, there exists $\delta>0$ such that
			\begin{equation*}
			\biggl|\mu(Q_1)(f)-\mu(Q_2)(f)\biggr|\leq\varepsilon+\frac{2\sup |f|}{\delta n}\cdot c.
			\end{equation*}
\end{proposition*}
We now consider two cases.
\begin{enumerate}
    \item Case 1: $\mu_{1,n}$ and $\mu_{2,n}$ denote the spectral measures associated to one of the matrices $A$, $D$, $K$. By Theorem~\ref{thmci} and by Proposition~6.7 in \cite{rgg}, for $f\in C_c(\mathbb{R},\mathbb{R})$ and for each $\varepsilon>0$ there exists $\delta>0$ such that
	\begin{equation*}
		\biggl|	\mu_{1,n}(f)-\mu_{2,n}(f)\biggr|\leq \varepsilon+\frac{2\sup|f|}{\delta n}\cdot 3c_1^2\cdot c_2.
			\end{equation*}Hence,
			\begin{equation*}
			    \lim_{n\rightarrow\infty}	\biggl|	\mu_{1,n}(f)-\mu_{2,n}(f)\biggr|\leq \varepsilon.
			\end{equation*}
    \item Case 2: $\mu_{1,n}$ and $\mu_{2,n}$ denote the spectral measures associated to $L$. Similarly to the first case, by Theorem~\ref{thmci} and by Proposition~6.7 in \cite{rgg}, for each $f\in C_c(\mathbb{R},\mathbb{R})$ and for each $\varepsilon>0$ there exists $\delta>0$ such that
	\begin{equation*}
		\biggl|	\mu_{1,n}(f)-\mu_{2,n}(f)\biggr|\leq \varepsilon+\frac{2\sup|f|}{\delta n}\cdot 2\sqrt{2n}\cdot c_1\cdot c_2.
			\end{equation*}Hence, as before,
			\begin{equation*}
			    \lim_{n\rightarrow\infty}	\biggl|	\mu_{1,n}(f)-\mu_{2,n}(f)\biggr|\leq \varepsilon.
			\end{equation*}This proves the claim.
\end{enumerate}
\end{proof}

\subsection{Strong convergence}\label{section:strong}
\begin{remark}
As shown in \cite{rgg}, the weak convergence in \cite[Theorem~6.4]{rgg} that we generalized in Theorem~\ref{thm:difference} cannot be substituted by the strong convergence in total variation distance. However, as proved in \cite[Cor.~6.11]{rgg}, the convergence in total variation distance holds, in the case of $L$, when considering growing families of graphs such that, for each $n$ even, $\Gamma_{1,n}$ is the disjoint union of two complete graphs on $n/2$ nodes, while $\Gamma_{2,n}$ is a ``connected sum'' of two complete graphs, namely, it is given by two copies of the complete graph on $n/2$ vertices, joined by at most $c$ edges, where $c=o(n)$. By Remark~\ref{rmk:graphs2}, this holds also for the $2$--complete hypergraphs in Def.~\ref{def:complete}. Here we generalize this result and we show that the strong convergence in total variation distance holds for various families of growing hypergraphs and with respect to all operators $A$, $D$, $K$ and $L$. While the proof of Lemma~6.9 and Corollary~6.20 in \cite{rgg} is based on the investigation of the eigenvectors of $L$ for the given graphs, here we prove some general spectral properties of symmetric matrices and we use them in order to prove our claim.\end{remark}
\begin{definition}
The total variation distance between two measures $\mu_1$ and $\mu_2$ on an interval $\textrm{I}\subset \mathbb{R}$ is
				\begin{equation*}
				d_{\textrm{tv}}(\mu_1,\mu_2):=\sup_{B\subseteq \textrm{I} \text{ measurable }}\biggl|\mu_{1}(B)-\mu_{2}(B)\biggr|.
				\end{equation*}
\end{definition}

\begin{lemma}\label{lemma:submatrix}
Let $Q$ be an $n\times n$ symmetric matrix, let $c\in\{1,\ldots,n-1\}$ and let $P$ be a submatrix of $Q$ of size $(n-c)\times(n-c)$. For each $\lambda\in\mathbb{R}$,
\begin{equation*}
    M_\lambda(P)\geq M_{\lambda}(Q)-c \qquad\text{ and }\qquad M_{\lambda}(Q)\geq M_\lambda(P)-c.
\end{equation*}
\end{lemma}
\begin{proof}
 By repeatedly applying the Cauchy Interlacing Theorem (Theorem~4.3.17 in \cite{frobenius}), for all $k\in\{1,\ldots,n-c\}$
    \begin{equation*}
        \lambda_{k}(Q)\leq \lambda_k(P)\leq \lambda_{k+c}(Q).
    \end{equation*}Therefore, 
    \begin{equation*}
        \lambda_k(Q)=\ldots =\lambda_{k+c}(Q)=\lambda \Longrightarrow \lambda_k(P)=\lambda
    \end{equation*}and similarly
    \begin{equation*}
\lambda_{k}(P)=\ldots=\lambda_{k+c}(P)=\lambda\Longrightarrow \lambda_{k+c}(Q)=\lambda.
\end{equation*}Hence, if $Q$ has eigenvalue $\lambda$ with multiplicity $M_{\lambda}(Q)$, then $P$ has eigenvalue $\lambda$ with multiplicity at least $M_\lambda(Q)-c$. That is, $M_\lambda(P)\geq M_{\lambda}(Q)-c$ for each $\lambda$. Similarly, $M_{\lambda}(Q)\geq M_\lambda(P)-c$.
\end{proof}
\begin{corollary}\label{cor:2c}
Let $Q_1$ and $Q_2$ be two $n\times n$ symmetric matrices that differ at most by $c$ rows (resp. columns). For each $\lambda\in\mathbb{R}$,
\begin{equation*}
    M_\lambda(Q_1)\geq M_{\lambda}(Q_2)-2c.
\end{equation*}
\end{corollary}
\begin{proof}
 Since $Q_1$ and $Q_2$ differ at most by $c$ rows (resp. columns), there exists a submatrix $P$ of both $Q_1$ and $Q_2$ that has size $(n-c)\times (n-c)$. By Lemma~\ref{lemma:submatrix},
 \begin{equation*}
  M_{\lambda}(Q_1)\geq M_\lambda(P)-c\geq M_{\lambda}(Q_2)-2c.
\end{equation*}
\end{proof}

\begin{theorem}\label{thm:tv}
Let $s\in\mathbb{N}$. For $n\in\mathbb{N}$, let $Q_{1,n}$ and $Q_{2,n}$ be two $n\times n$ symmetric matrices that differ at most by $c=o(n)$ rows (resp. columns). Assume that, for each $n$, there exist at most $s$ eigenvalues of $Q_{1,n}$ whose sum of multiplicities is at least $n-k$, where $k=o(n)$. Then,
\begin{equation*}
d_{\textrm{tv}}(\mu(Q_{1,n}),\mu(Q_{2,n}))\xrightarrow{n\rightarrow\infty}0.
\end{equation*}
\end{theorem}
\begin{proof}
 By assumption, given $n$ there exist at most $s$ eigenvalues $a_{n,j}$ of $Q_{1,n}$ with respective multiplicities $q_{n,j}$, such that $\sum_{j}q_{n,j}\geq n-k$. Hence, we can write
\begin{equation*}
    \mu(Q_{1,n})=\sum_{j}\frac{q_{n,j}}{n}\cdot \delta_{a_{n,j}}+\sum_{\eta}\frac{1}{n}\delta_\eta,
\end{equation*}where the second sum is over at most $k$ eigenvalues $\eta$ of $Q_{1,n}$. By Corollary~\ref{cor:2c},
\begin{equation*}
    \mu(Q_{2,n})=\sum_{j}\frac{q_{n,j}-2c}{n}\cdot \delta_{a_{n,j}}+\sum_{\eta}\frac{1}{n}\delta_\eta,
\end{equation*}where the sum is over at most $k+2cs$ eigenvalues $\eta$ of $Q_{2,n}$. Therefore,
\begin{equation*}
d_{\textrm{tv}}(\mu(Q_{1,n}),\mu(Q_{2,n}))\leq \frac{k+2cs}{n}
\end{equation*}tends to zero for $n\rightarrow\infty$, since by assumption $c=o(n)$ and $k=o(n)$.
\end{proof}
\begin{corollary}\label{cor:tv}
Let $s\in\mathbb{N}$. For $n\in\mathbb{N}$, let $Q_{1,n}$ and $Q_{2,n}$ be two $n\times n$ symmetric matrices that differ at most by $c=o(n)$ rows (resp. columns). Assume that, for each $n$, $Q_{1,n}$ has at most $s$ distinct eigenvalues. Then, 
\begin{equation*}
d_{\textrm{tv}}(\mu(Q_{1,n}),\mu(Q_{2,n}))\xrightarrow{n\rightarrow\infty}0.
\end{equation*}
\end{corollary}

As a consequence of Corollary~\ref{cor:tv}, we can prove convergence in total variation distance for spectral measures for various growing families of hypergraphs.
\begin{corollary}Let $k,r\in\mathbb{N}$ with $r\geq 2$. For each $n\in\mathbb{N}_{\geq 2}$, let $\Gamma_{1,n}$ be the disjoint union of $k$ $r$--complete hypergraphs on $n$ nodes, and let $\Gamma_{2,n}$ be a hypergraph that differs from $\Gamma_{1,n}$ by at most $c_1$ hyperedges of cardinality at most $c_2$, where $c_1\cdot c_2=o(n)$. Denote by $\mu_{1,n}$ and $\mu_{2,n}$ the corresponding spectral measures with respect to one of the matrices $A$, $D$, $K$, $L$. Then,
\begin{equation*}
    d_{\textrm{tv}}(\mu_{1,n},\mu_{2,n})\xrightarrow{n\rightarrow\infty}0.
\end{equation*}
\end{corollary}
\begin{proof}Since $\Gamma_{1,n}$ and $\Gamma_{2,n}$ differ by at most $c_1$ hyperedges of cardinality at most $c_2$, their associated operators differ at most by $c_1\cdot c_2$ rows (resp. columns). The claim follows from Lemma~\ref{lemma:complete} and Corollary~\ref{cor:tv} with $c=c_1\cdot c_2$.
\end{proof}
Similarly, in Corollary~\ref{cor:flowers1} and Corollary~\ref{cor:flowers2} below we prove strong convergence for growing families of hyperflowers, in the settings of Proposition~\ref{prop:flowers1} and Proposition~\ref{prop:flowers2}, respectively.
\begin{corollary}\label{cor:flowers1}
Fix $t,l\in\mathbb{N}$. For each $n \geq tl+1$, let $\Gamma_{n,1}$ be the $l$--hyperflower with $t$ twins on $n$ vertices and let $\Gamma_{n,2}$ be a hypergraph that differs from $\Gamma_{n,1}$ by at most $c_1$ hyperedges of cardinality at most $c_2$. Denote by $\mu_{1,n}$ and $\mu_{2,n}$ the corresponding spectral measures with respect to one of the matrices $A$, $D$, $K$, $L$. Then,
\begin{equation*}
    d_{\textrm{tv}}(\mu_{1,n},\mu_{2,n})\xrightarrow{n\rightarrow\infty}0.
\end{equation*}
\end{corollary}
\begin{proof}It follows from Theorem~\ref{thm_spectrumflowers} and Corollary~\ref{cor:tv}.
\end{proof}
\begin{corollary}\label{cor:flowers2}
Fix $t,c\in \mathbb{N}$. For each $l\in \mathbb{N}$, let $\Gamma_{c+tl,1}$ be the $l$--hyperflower with $t$ twins on $n$ vertices and let $\Gamma_{c+tl,2}$ be a hypergraph that differs from $\Gamma_{c+tl,1}$ by at most $c_1$ hyperedges of cardinality at most $c_2$. Denote by $\mu_{1,c+tl}$ and $\mu_{2,c+tl}$ the corresponding spectral measures with respect to one of the matrices $A$, $D$, $K$, $L$. Then, 
\begin{equation*}
    d_{\textrm{tv}}(\mu_{1,c+tl},\mu_{2,c+tl})\xrightarrow{l\rightarrow\infty}0.
\end{equation*}
\end{corollary}
\begin{proof}It follows from Theorem~\ref{thm_spectrumflowers} and Corollary~\ref{cor:tv}.
\end{proof}
\begin{remark}As shown in \cite{rgg}, if for $n$ even we let $\Gamma_{n,1}$ be the path on $n$ vertices and we let $\Gamma_{n,2}$ be the disjoint union of two paths on $n/2$ vertices, then the total variation distance between the two measures with respect to $L$ does not tend to zero as $n\rightarrow\infty$. In contrast to the above examples, such growing families of paths do not satisfy Corollary~\ref{cor:tv} with respect to $L$ because all eigenvalues of the paths have multiplicity $1$. The same holds for the case of growing graph cycles that we investigated in Corollary~\ref{cor:cycle-graphs} with respect to $D$, $A$, $K$ and $L$.
\end{remark}
 
\section*{Acknowledgments}The author is grateful to J\"urgen Jost (MPI MiS) for the interesting discussions that have been of inspiration for this paper; to Eleonora Andreotti (Chalmers University) and to Emanuele Convergologo Martinuzzi (University of Bonn) for the helpful comments. Thanks to the anonymous referees for the suggestions for improvements.


\begin{thebibliography}{99}

\bibitem{andreotti2020eigenvalues} E. Andreotti, Spectra of hyperstars on public transportation networks, \emph{arXiv preprint}, arXiv:2004.07831 (2020).


\bibitem{AndreottiMulas} E. Andreotti and R. Mulas, Spectra of Signless Normalized Laplace Operators for Hypergraphs, \emph{arXiv preprint}, arXiv:2005.14484v1 (2020).

\bibitem{orientedhyp2018} G. Chen, V. Liu, E. Robinson, L. J. Rusnak and K. Wang, A characterization of oriented hypergraphic Laplacian and adjacency matrix coefficients, \emph{Linear Algebra Appl.} 556 (2018), 323--341.

\bibitem{Chung} F. Chung, \emph{Spectral graph theory}, American Mathematical Society, 1997.

\bibitem{orientedhyp2019} L. Duttweiler and N. Reff, Spectra of cycle and path families of oriented hypergraphs, \emph{Linear Algebra Appl.} 578 (2019), 251--271.

\bibitem{orientedhyp2019-3} W. Grilliette, J. Reynes and L. J. Rusnak, Incidence Hypergraphs: Injectivity, Uniformity, and Matrix-tree Theorems, \emph{arXiv preprint}, arXiv:1910.02305 (2019).

\bibitem{Gu} J. Gu, \emph{The spectral distance based on the normalized Laplacian and applications to large networks}, PhD thesis, University of Leipzig, 2014.

\bibitem{spectraldistances} J. Gu, B. Hua and S. Liu, Spectral distances on graphs, \emph{Discrete Appl. Math.} 190-191 (2015), 56--74.

\bibitem{JJspectralclasses} J. Gu, J. Jost, S. Liu and P. F. Stadler, Spectral classes of regular, random, and empirical graphs, \emph{Linear Algebra Appl.} 489 (2016), 30--49.

\bibitem{frobenius} R. A. Horn and C. R. Johnson, \emph{Matrix Analysis}, Cambridge University Press, second edition, 2013.

\bibitem{Hypergraphs}  J. Jost and R. Mulas, Hypergraph Laplace operators for chemical reaction networks, \emph{Adv. Math.} 351 (2019), 870--896.

\bibitem{pLaplacian} J. Jost, R. Mulas and D. Zhang, $p$-Laplace Operators for Chemical Hypergraphs, \emph{arXiv preprint}, arXiv:2007.00325 (2020).

\bibitem{orientedhyp2019-2} O. Kitouni and N. Reff, Lower bounds for the Laplacian spectral radius of an oriented hypergraph, \emph{Australas. J. Combin.} 74(3) (2019), 408--422.

\bibitem{rgg} A. Lerario and R. Mulas, Random geometric complexes and graphs on Riemannian manifolds in the thermodynamic limit, \emph{Discrete Comput. Geom.} (2020), DOI:https://doi.org/10.1007/s00454-020-00238-4.

\bibitem{Sharp}R. Mulas, Sharp bounds for the largest eigenvalue of the normalized hypergraph Laplace Operator, \emph{Math. Notes} (2021), To appear.

\bibitem{Master-Stability} R. Mulas, C. Kuehn and J. Jost, Coupled Dynamics on Hypergraphs: Master Stability of Steady States and Synchronization, \emph{Phys. Rev. E}, 101(6) (2020), 062313.

\bibitem{MulasZhang} 
R. Mulas and D. Zhang, Spectral theory of Laplace Operators on chemical hypergraphs, \emph{arXiv preprint}, arXiv:2004.14671 (2020).

\bibitem{orientedhyp2014}N. Reff, Spectral properties of oriented hypergraphs, \emph{Electron. J. Linear Algebra} 27 (2014).

\bibitem{orientedhyp2016} N. Reff, Intersection graphs of oriented hypergraphs and their matrices, \emph{Australas. J. Combin.} 65(1) (2016), 108--123.

\bibitem{ReffRusnak} N. Reff and L. Rusnak,  An oriented hypergraphic approach to algebraic graph theory, \emph{Linear Algebra Appl.} 437 (2012), 2262--2270.

\bibitem{orientedhyp2017} E. Robinson, L. Rusnak, M. Schmidt and P. Shroff, Oriented hypergraphic matrix-tree type theorems and bidirected minors via Boolean order ideals, \emph{J. Algebraic Combin.} (2017).

\bibitem{orientedhyp2013} L. Rusnak, Oriented Hypergraphs: Introduction and Balance, \emph{Electron. J. Combin.} 20(3) (2013).

\bibitem{Shi92} C.-J. Shi, A signed hypergraph model of the constrained via minimization problem, \emph{Microelectron. J.} 23(7) (1992), 533--542.

\bibitem{Tao} T. Tao, \emph{Topics in random matrix theory}, Graduate Studies in Mathematics 132, American Mathematical Society, Providence, RI, 2012.

\end{thebibliography}
\end{document}